\documentclass[11pt,a4paper]{article}
\usepackage{jheppub}
\usepackage{amsmath}
\usepackage[most]{tcolorbox}
\usepackage{dsfont}
\usepackage{ulem}
\usepackage{natbib}
\usepackage{xcolor}
\usepackage[hang,flushmargin]{footmisc}
\usepackage{amsthm}
\newtheorem{theorem}{Theorem}[section]
\newtheorem{lemma}[theorem]{Lemma}
\newtheorem{definition}{Definition}[section]
\usepackage{tikz-cd}
\usepackage{enumitem}
\usepackage{amsfonts,amssymb, amscd,amsmath,latexsym,amsbsy,bm}
\usepackage{stmaryrd}
\usepackage{todonotes}

\makeatletter\renewcommand{\@biblabel}[1]{#1.}\makeatother

\newtcolorbox{empheqboxed}{colback=gray!20, 
	colframe=white,
	width=\textwidth,
	sharpish corners,
	top=0mm, % default value 2mm
	bottom=0pt
}

\title{Bailey Pairs for the Tetrahedron Index}

\author{Ilmar Gahramanov$^{a,c,d}$, Sinan Ulaş Öztürk$^{a}$, Uveys Turhan$^{a,b}$}

\affiliation{$^a$ Department of Physics, Bogazici University,\\ 34342 Bebek, Istanbul, Türkiye\\[-0.5cm]

$^b$ Department of Mathematics, Bogazici University,\\ 34342 Bebek, Istanbul, Türkiye\\[-0.5cm]
	
$^c$ Center for Mathematics and its Applications, Khazar University, \\ Mehseti St. 41, AZ1096, Baku, Azerbaijan\\[-0.5cm]

$^d$ Steklov Mathematical Institute of Russian Academy of Sciences, 119991 Moscow, Russia}

\emailAdd{ilmar.gahramanov@bogazici.edu.tr, ulasoztrk@gmail.com, uveys.turhan@std.bogazici.edu.tr}

\abstract{In this work, we develop new Bailey pairs for the pentagon identity satisfied by the tetrahedron index, expressible in terms of $q$-series. Since the tetrahedron index underlies topological invariants of 3-manifolds and related knots, our construction may offer a new framework to deriving knot invariants through the Bailey chain.}


\begin{document}
	\maketitle
	\flushbottom
	
\section{Introduction}

Bailey's lemma is a useful method for systematically deriving hypergeometric identities. The Bailey construction also appears in computations of supersymmetric gauge theories. Its relevance was first recognized in \cite{Berkovich:1995xg,andrews1998trinomial}, and later extended to supersymmetric quiver gauge theories\footnote{It also appears in conformal theories; see, e.g., \cite{Foda:1994wz,Chim:1996rc,Berkovich:1996hb}.} \cite{Kashaev:2012cz,Gahramanov:2015cva,Brunner:2017lhb,Gahramanov:2021pgu,Gahramanov:2022jxz,Catak:2022glx,Akbulut:2025kow}. 

In this work, we construct novel Bailey pairs for the pentagon identity for the tetrahedron index, which can be written in terms of q-hypergeometric functions. Similar Bailey pair constructions for different integral pentagon identities have been previously studied in \cite{Gahramanov:2021pgu,Gahramanov:2022jxz}.

Let $q, z \in \mathbb{C}$ with $|q| < 1$. We define the q-Pochhammer symbol
\begin{gather}\label{q-poch}
    (z;q)_{n}:=\prod^{n-1}_{k=0}(1-zq^k)
\end{gather}
The tetrahedron index ${I_{\triangle}}(m,e)$, first introduced in \cite{ Dimofte:2011py}, and further studied in \cite{Garoufalidis2012The3I, GaroufalidisYu_2024_3DIndexSkein, GaroufalidisKashaev_2019_Meromorphic3DIndex} is given by the following expression
\begin{equation}\label{tetindex}
{I_{\triangle}}(m,e) :=\sum_{n=\frac{1}{2}(|e|-e)}^{\infty}\frac{(-1)^n q^{\frac12 n (n+1)-(n+\frac12 e)m}}{(q;q)_n (q; q)_{n+e}} \;
\end{equation}

The tetrahedron arises in the context of superconformal index for  three-dimensional $\mathcal N=2$ supersymmetric gauge theories, see, e.g. \cite{Gahramanov:2014ona}. Serving as a fundamental building block for the topological index associated with the ideal triangulation of $3$-manifolds, it was introduced in \cite{Dimofte:2011ju, Dimofte:2011py} as an invariant of compact, orientable $3$-manifolds with non-empty boundary. One can also employ the tetrahedron index to construct knot invariants. For instance, for the \( 4_1 \) knot, the corresponding invariant is given in \cite{Dimofte:2011py, Garoufalidis:2022mbi, Garoufalidis:2023wez}
\[
\text{Ind}_{4_1}(q) = \sum_{k_1, k_2 \in \mathbb{Z}} I_{\Delta}(k_1, k_2) \, I_{\Delta}(k_2, k_1)
= 1 - 8q - 9q^{2} + 18q^{3} + 46q^{4} + \cdots .
\]

\begin{theorem}
Let $m_i, e_i \in \mathbb{Z}$, then ${I_{\triangle}}(m,e)$  obeys the pentagon identity  \cite{Dimofte:2011py, Garoufalidis2012The3I}:
    \begin{equation}
         \label{pentagon}
{I_{\triangle}}(m_1-e_2,e_1){I_{\triangle}}(m_2-e_1,e_2)
 = \ \sum_{e_3} q^{e_3} {I_{\triangle}}(m_1,e_1+e_3){I_{\triangle}}(m_2,e_2+e_3){I_{\triangle}}(m_3,e_3)
\end{equation}
with the balancing condition is
\begin{equation}\label{condition}
 \newline
 m_1 + m_2 = m_3
\end{equation}

\end{theorem}
Note that the pentagon identity (\ref{pentagon}) represents the mirror symmetry between $d=3$, ${\mathcal N}=2$ supersymmetric quantum electrodynamics with $U(1)$ gauge group and $N_f = 1$ in IR fixed point and its mirror partner, the free Wess-Zumino theory (also known as XYZ model) \cite{Dimofte:2011py}. From a topological perspective, the pentagon identity can be interpreted as a Pachner $3-2$ move for triangulated $3$-manifolds, as an example, different gluings of the tetrahedra into the bipyramid lead to different UV descriptions of the same theory,  as mentioned in \cite{Dimofte:2011py, Kashaev:2012cz}.

The tetrahedron index has the triality property, proven in the appendix of \cite{Dimofte:2011py}
\begin{equation}\label{triality}
    {I_{\triangle}}(m,e) = (-q^{\frac{1}{2}})^m {I_{\triangle}}(-e-m,e) = (-q^{\frac{1}{2}})^e {I_{\triangle}}(e,-e-m)
\end{equation}
Before constructing the Bailey pairs, we need to manipulate (\ref{pentagon}) to obtain a more convenient form. Let $e_0 \in \mathbb{Z}$, performing the shifts: $e_3 \xrightarrow{} e_3 + e_0$,  $e_1 \xrightarrow{} e_1 - e_0$, $e_2 \xrightarrow{} e_2 - e_0$ on (\ref{pentagon}) and using the triality property on the left-hand side, give the following form of the pentagon identity, which will be used for the construction of Bailey pairs in the next section
\begin{multline}\label{pentagonfinal}
    -q^{-\frac{3}{2}e_0}{I_{\triangle}}(m_1 - e_2 + e_0,e_1 - e_0){I_{\triangle}}(-m_2 + e_1 - e_2, m_2 - e_1 + e_0)= \\ 
    \ \sum_{e_3} q^{e_3 + \frac{1}{2}(m_2 - e_1)} {I_{\triangle}}(m_1, e_1 + e_3) {I_{\triangle}}(m_2, e_2 + e_3) {I_{\triangle}}(m_1 + m_2 ,e_0 + e_3)  \;. \end{multline}

%%%%%%%%%%%%%%%%%%%%%%%%%%%%%%%%%%%%%%%%%%%%%%
\section{Bailey Pairs}
%%%%%%%%%%%%%%%%%%%%%%%%%%%%%%%%%%%%%%%%%%%%%%

Introduced by Bailey, Bailey Lemma is a useful way to derive Rogers-Ramanujan type identities \cite{Bailey1946SomeII, Bailey1948IdentitiesOT}. Following Bailey's work, Andrews discovered an iterative method to derive new pairs from a known pair\footnote{ Spiridonov generalized the method for integral identities \cite{Spiridonov2003ABT, Spiridonov2002AnEI}.} \cite{Andrews1984MULTIPLESR}. The the generalized version of the Bailey chain is a couple of infinite sequences of holomorphic functions $\{\alpha^{(i)}_n\}_{n \geq 0}$ and $\{\beta^{(i)}_n\}_{n \geq 0}$ such that there exists an identity independent of \textit{i} which connect $\alpha^{(i)}_n$ and $\beta^{(i)}_n$ as
\begin{equation}
    \beta^{(i)}_n=F_n(\alpha^{(i)}_0, \alpha^{(i)}_1,..., \alpha^{(i)}_n) \; \text{,}
\end{equation}
where $F$ can be an operator which may now include sum or integrals. Here, $\alpha^{(i)}_n$ and $\beta^{(i)}_n$ are constructed according to
\begin{align}
    \alpha^{(i)}_n=G(\alpha^{(i)}_0, \alpha^{(i)}_1,..., \alpha^{(i)}_{n-1}) \; \text{,}\\
    \beta^{(i)}_n=H(\beta^{(i)}_0, \beta^{(i)}_1,..., \beta^{(i)}_{n-1}) \; \text{,}
\end{align}
where $G$ and $H$ represent integral-sum operators.
\begin{definition}
Let $\{\alpha_m(t)\}_{m \in \mathbb{Z}}$ and $\{\beta_m(t)\}_{m \in \mathbb{Z}}$ be two sequences of functions. They are said to form a Bailey pair with respect to the parameter $t$ if
\begin{align}\label{baileydef1}
        \beta_k (t)= \sum_{n \in \mathbb{Z}} I_{\triangle}(t, n + k) \alpha_n(t)
\end{align}
\end{definition}
\begin{lemma}
If $\{\alpha_m(t)\}_{m \in \mathbb{Z}}$ and $\{\beta_m(t)\}_{m \in \mathbb{Z}}$ form a Bailey pair with respect to $t$, then the following sequences
\begin{align}
            & \alpha'_n(t + s)= -q^{-\frac{3}{2}p} I_{\triangle}(3t - s + n,2s - t  - n)\alpha_n(t) \label{baileydef_a} \\ 
            & \nonumber \beta'_m(s + t)=\sum_{k \in \mathbb{Z}}q^{k + \frac{m}{2}}  I_{\triangle}(-m - 2s + 2t,2s - t + k) \\
            & \qquad \qquad \qquad \qquad \times I_{\triangle}(m + 2s - t,k - m - s - t) \beta_k(t)
            \label{baileydef2}
\end{align}
\vspace{-0.5cm}
form a Bailey pair with respect to the parameter $t + s$.
\end{lemma}

\begin{proof} We need to show that
\begin{align}
     \beta'_k (s + t)= \sum_{n \in \mathbb{Z}} I_{\triangle}(t + s, n + k) \alpha'_n(s + t)
     \label{baileydef3}
\end{align}
Inserting (\ref{baileydef1}) in (\ref{baileydef2}), we first calculate the left-hand side of the equality (\ref{baileydef3})
\begin{multline} \label{bpr}
    \beta'_m(s + t)= 
    \sum_{n \in \mathbb{Z}} \sum_{k \in \mathbb{Z}}q^{k + \frac{m}{2}}  I_{\triangle}(-m - 2s + 2t,2s - t + k)\\ \times I_{\triangle}(m + 2s - t,k -m - s - t)  I_{\triangle}(t, n + k) \alpha_n(t) \end{multline}

Upon renaming the variables as
\begin{align}
    &m = m_2 - e_1  \qquad \qquad \quad \qquad \; \; \;
    2t -m - 2s = m_1 \\ 
    &2s - t + k = e_1 + e_3   \qquad \; \; \; \; \; \; \quad
    m + 2s - t = m_2 \\
    &k - m - s - t= e_2 + e_3  \; \; \; \;  \qquad
    t = m_1 + m_2 \\
    &n + k = e_0 + e_3
\end{align}
we identify the right-hand side of the pentagon identity (\ref{pentagonfinal}). After inserting (\ref{baileydef_a}) and the left-hand side of (\ref{pentagonfinal}) in (\ref{bpr}), we get the desired equality.
\begin{gather} \nonumber
    \beta'_m(s + t)=  \sum_{n \in \mathbb{Z}} -q^{-\frac{3}{2}n}I_{\triangle}(3t - s + n,2s - t  - p)I_{\triangle}(t+s, n+m)\alpha_n(t) \\
       \xrightarrow{} \beta'_m(s + t) =  \sum_{n \in \mathbb{Z}} I_{\triangle}(t+s, n+m)\alpha'_n(t)
\end{gather}
\end{proof}

\section{Conclusions}

In this work, we have constructed new Bailey pairs associated with the pentagon identity of the tetrahedron index. The construction follows a similar approach to that developed in \cite{Gahramanov:2021pgu}, and it would be interesting to extend the Bailey framework to other dualities discussed in \cite{Gahramanov:2022jxz,Bozkurt:2020gyy} in terms of the tetrahedron index. This approach may provide a systematic way to uncover new supersymmetric quiver dualities.

We would like to note that a possible direction for future research is the construction of the star–triangle relation for certain integrable lattice models using these Bailey pairs.

\section*{Acknowledgements} 

We would like to thank all participants of the seminar series “Istanbul Integrability Initiative” (https://istringy.org) at Boğaziçi University (Istanbul, Türkiye). Ilmar Gahramanov is supported by the 3501–TÜBİTAK Career Development Program under grant number 122F451. The work of Ilmar Gahramanov is also partially supported by the Russian Science Foundation under grant number 22-72-10122. Ilmar Gahramanov would like to thank the Nesin Mathematics Village (Şirince, Türkiye) for its hospitality, where part of this work was carried out. Sinan Ulaş Öztürk is supported by the 2209-A TÜBİTAK National/International Research Projects Fellowship Programme for Undergraduate Students under grant number 1919B012430350.

	\bibliographystyle{JHEP} % bibliography
	\bibliography{references}
\end{document}